\documentclass[12pt]{article}
\usepackage{mycommands}
\usepackage{fullpage}
\usepackage{faktor}
\newcommand{\toggle}[1]{\triangle\left\{#1 \right\} }

\begin{document}
\pagestyle{plain}

\title{Transitive generalized toggle groups containing a cycle}
\date{\today\\[10pt]}
\author{Jonathan S. Bloom\\
Dan Saracino}
\maketitle

\begin{abstract}
    In \cite{striker2018rowmotion} Striker generalized   Cameron and Fon-Der-Flaass's notion of a toggle group.  In this paper we begin the study of transitive generalized toggle groups that contain a cycle.  We  first show that if such a group has degree $n$ and contains a transposition or a 3-cycle then the group contains $A_n$. Using the result about transpositions, we then prove that a transitive generalized toggle group that contains a short cycle must be primitive. Employing a result of Jones \cite{jones2014primitive}, which relies on the classification of the finite simple groups, we conclude that any transitive generalized toggle group of degree $n$ that contains a cycle with at least 3 fixed points must also contain $A_n$. Finally, we look at imprimitive generalized toggle groups containing a long cycle and show that they decompose into a direct product of primitive generalized toggle groups each containing a long cycle.  
\end{abstract}
\section{Introduction}

The study of toggle groups dates back to the 1995 work \cite{cameron1995orbits} of Cameron and Fon-Der-Flaass on order ideals of posets.  To provide a simpler proof of a result from \cite{fon1993orbits}, Cameron and Fon-Der-Flaass associated to each element $p$ of a finite poset $P$ a permutation $s_p$ of the set $J(P)$ of order ideals of $P$.  For each $I\in J(P)$, $s_p(I)$ was defined to be $I\toggle p$ if $I\toggle p \in J(P)$, and $s_p(I)=I$ otherwise.  After using the $s_p$'s to achieve their proof, they went on to study the subgroup $G(P)= \langle s_p\ |\ p\in P\rangle$ of the symmetric group $S_{J(P)}$.  They proved \cite[Theorem 4]{cameron1995orbits} that if the Hasse diagram of $P$ is connected then $G(P)$ contains the alternating group on $J(P)$, while if $P$ is the disjoint union of posets $P_1$ and $P_2$ then $G(P)\cong G(P_1)\times G(P_2)$.

In \cite{striker2012promotion} Striker and Williams used the $s_p$'s in studying what they called ``promotion" and ``rowmotion" on order ideals.  They called the $s_p$'s \emph{toggles} and $G(P)$ the \emph{toggle group}. Several other papers soon followed \cite{einstein2014piecewise,grinberg2016iterative, gringberg2015iterative} that used analogues of the $s_p$'s in different contexts. In \cite{striker2018rowmotion} Striker noted that the definition of the $s_p$'s only relied on the order ideals of $P$ being subsets of $P$, and generalized the definition of the toggle group.

\begin{definition}
Let $E$ be a finite set, which we refer to as the \emph{ground set}. Let $\cL$ be a set of subsets of $E$. For each $e\in E$, define the \emph{toggle} $\tau_e: \cL\rightarrow \cL$ by letting $\tau_e(X)=X\toggle e$ if $X\toggle e\in\cL$ and $\tau_e(X)=X$ otherwise.  The \emph{generalized toggle group} $T(\cL)$ is the subgroup of the symmetric group $S_{\cL}$ generated by $\makeset{\tau_e}{e\in E}$.

\end{definition}

Striker initiated the study of $T(\cL)$ for a number of combinatorially interesting choices of $\cL$. In particular, she obtained analogs of \cite[Theorem 4]{cameron1995orbits}, for several examples where $\cL$ is a set of subsets of a finite poset or graph.

In what follows, we head in a somewhat different direction. There is a long history of results about primitive permutation groups that contain a nontrivial cycle, dating back to the work of Jordan in the 1870's.  We will prove some structure results for all transitive generalized toggle groups that contain a nontrivial cycle.  In Section 3 we will show that if $T(\cL)$ is transitive with degree $n$ and contains a transposition then $T(\cL)$ must be the full symmetric group $S_n$, and if $T(\cL)$ is transitive with degree $n$ and contains a 3-cycle then $T(\cL)$ must contain the alternating group $A_n$. These results are analogues of standard results \cite[Theorem 3.3A]{dixon1996permutation} for primitive permutation groups, but we will obtain them for transitive generalized toggle groups without assuming primitivity. Using the result for transpositions, we will then show that if $T(\cL)$ is transitive with degree $n$ and contains a ``short cycle" (i.e., a cycle of length at most $n-1$), then $T(\cL)$ is primitive. Using \cite[Corollary 1.3]{jones2014primitive}  (a result that relies on the classification of the finite simple groups) it will then follow that if $T(\cL)$ contains a cycle with at least three fixed points then $T(\cL)$ contains $A_n$.

In Section 4 we will show that if $T(\cL)$ is imprimitive and contains a nontrivial cycle (necessarily a ``long cycle" of length $n$) then $T(\cL)$ is isomorphic to a direct product of primitive groups $T(\cL_1) \times \cdots \times T(\cL_k)$ where each $T(\cL_i)$ has degree at least two and contains a long cycle whose length is the degree of $T(\cL_i)$.

In obtaining the results of Section 4 we will use what Striker \cite{striker2018rowmotion} calls the ``toggle-disjoint Cartesian product" of sets $\cL_1$ and $\cL_2$.

\begin{definition}
If $E_1$ and $E_2$ are disjoint ground sets and $\cL_i\subseteq 2^{E_i}$ for $i=1,2$  then we say \emph{$\cL$ is the toggle-disjoint Cartesian product of $\cL_1$ and $\cL_2$} and write $\cL=\cL_1 \otimes \cL_2$ if $\cL=\makeset{X_1\cup X_2}{X_i\in \cL_i}$.

\end{definition}

Striker makes this definition  under an assumption weaker than $E_1\cap E_2=\emptyset$, but we will only use it when $E_1\cap E_2=\emptyset$.  It is straightforward to check \cite[Theorem 2.18]{striker2018rowmotion} that if $\cL$ is the toggle-disjoint Cartesian product of $\cL_1$ and $\cL_2$ then $T(\cL)\cong T(\cL_1)\times T(\cL_2)$.

\begin{conventions*}
It may happen that for some elements $e\in E$, $\tau_e$ is the identity element of the symmetric group on $\cL$.  Since we will always be assuming that $T(\cL)$ is transitive in what follows, this can only happen if $e$ is in all or none of the sets in $\cL$.  We will always remove such $e$'s from $E$. Doing this has no effect on $T(\cL)$.

To simplify terminology we shall refer to generalized toggle groups simply as toggle groups.    
\end{conventions*}

\section{Block systems in toggle groups}

Recall that if a group $G$ acts on a set $X$ then we say a subset $\cB\subseteq X$ is a \emph{block} for $G$ if, for all $g\in G$, either $g(\cB) = \cB$ or $g(\cB)$ is disjoint from $\cB$.  If $G$ acts transitively on $X$ and $\cB$ is a block for $G$ then the sets $\makeset{g(\cB)} {g\in G}$ form a partition of $X$. We call such a partition a \emph{block system} for $G$. Clearly, the blocks in a given block system must all have the same size. If there exists a block system for which this size is strictly between 1 and $|X|$ then the group action is said to be \emph{imprimitive}, and otherwise the action is said to be \emph{primitive}. A good source for background information on these ideas is  \cite[Chapter 1]{dixon1996permutation}.

The purpose of this section is to explore how toggles behave on a block system. In particular, the lemmas established in this section describe a rich structure in the context of imprimitive toggle groups.  We then exploit this structure in the following sections.  

Recall that, as in the introduction, $E$ is our finite ground set, $\cL$ is a collection of subsets of $E$ and $T(\cL)$ is the corresponding toggle group.  We assume throughout that $T(\cL)$
acts transitively on $\cL$, and we fix a system of blocks for $T(\cL)$. We  often denote individual blocks by $\cB$ or $\cC$.

\begin{definition}
 We say that $g\in T(\cL)$ is \emph{type 2} if $g(\cB) = \cB$ for all blocks $\cB$, and that $g$ is \emph{type 1} otherwise.  
\end{definition}

\begin{lemma}\label{lem:commuting toggles}
Let $\tau_y$ be a type 1 toggle in  $T(\cL)$.  
\begin{enumerate}
    \item[i)] If $\tau_y(\cB) = \cB$ for some block $\cB$ then $\tau_y$ is the identity on this block.
    \item[ii)] $\tau_y$ commutes with all type 2 toggles in $T(\cL)$.  
    \item[iii)] For any type 2 toggle $\tau$ and blocks $\cB, \cC$, the cycle structure of $\tau|_{\cB}$ is the same as the cycle structure of $\tau|_{\cC}$.
\end{enumerate}
\end{lemma}

\begin{proof}
    We prove i) by contradiction.  Suppose that for some $A\in \cB$ we have $A, A\toggle{y}\in \cB$. By definition of $\tau_y$ there exists another block $\cC \neq \cB$.   By transitivity there exist toggles $\tau_{z_1},\ldots,\tau_{z_k}$ such that $\tau_{z_k}\cdots \tau_{z_1}(\cB)=\cC$. Let $\cB_0=\cB$ and $\cB_i=\tau_{z_i}(\cB_{i-1})$. By omitting some $\tau_i$'s we can assume that $\cB_i\neq \cB_{i-1}$.  Then $A\toggle{z_1}, A\toggle{y,z_1}\in \cB_1$ and hence $\tau_y(A\toggle{z_1}) = A\toggle{z_1,y}\in \cB_1$.  So $\tau_y(\cB_1) = \cB_1$. By iterating this argument we can show that $\tau_y(\cC)=\cC$.  Since $\cC$ was an arbitrary block distinct from $\cB$, it follows that $\tau_y$ is type 2, a contradiction. 
    
    To prove ii) let $\tau_x$ be a type 2 toggle and fix a block $\cB$.  If $\tau_y(\cB) = \cB$ then by i) it is clear that $\tau_x$ and $\tau_y$ commute at all elements of $\cB$.  So we may assume that $\tau_y(\cB) = \cC$ where $\cB \neq \cC$.  Fix $A \in \cB$ so that $A\toggle{y} = \tau_y(A) \in \cC$. If $\tau_x(A) = A$ then $A\toggle{x}\not\in \cB$ (since $A\toggle{x}\not\in \cL$) and hence $A\toggle{x,y}\notin \cC$.  So 
$$\tau_y\tau_x(A) = \tau_y(A) = A\toggle{y} = \tau_x(A\toggle{y}) = \tau_x\tau_y(A),$$  where the third equality follows from the fact that $\tau_x$ is type 2 so that $\tau_x(A\toggle{y})$ must be in $\cC$. So if $\tau_x(A)=A$ then $\tau_x$ and $\tau_y$ commute at $A$. If $\tau_x(A)= A\toggle{x}$ then $A\toggle{x}\in \cB$ and therefore $\tau_y\tau_x(A)= A\toggle{x,y} = \tau_x\tau_y(A)$. 

To prove iii) consider a type 2 toggle $\tau$ and blocks $\cB, \cC$.  As in the proof of i) it suffices by transitivity to prove the claim in the case where $\tau_y(\cB) = \cC$. If we write  $\tau$ as the product of its block restrictions we have
$$\tau = \cdots \tau|_\cB\cdots\tau|_\cC \cdots.$$
Conjugating by $\tau_y$ gives
$$\tau_y\tau\tau_y = \cdots (\tau_y\tau|_\cB\tau_y)\cdots (\tau_y\tau|_\cC\tau_y) \cdots$$
and we note that
$\tau_y\tau|_{\cC}\tau_y$ is a permutation of $\cB$.  Consequently $$\tau_y\tau|_{\cC}\tau_y = (\tau_y\tau\tau_y)|_{\cB} = \tau|_\cB,$$
where the last equality follows by part ii) and the fact that $\tau_y$ is type 1. Since conjugation preserves cycle structure the result is now clear.
\end{proof} 

\begin{lemma}\label{lem:type 2 toggles}
     Let $\rho\in T(\cL)$ be a product of type 2 toggles.  If $\rho$ fixes all elements, pointwise, in some block $\cB$ then $\rho=1$.  
\end{lemma}
\begin{proof}
    If $\cB$ is the only block we are done.  Otherwise, transitivity implies that there must be some type 1 toggle $\tau$ such that $\tau(\cB) = \cC$ with $\cB \neq \cC$.  It follows by Lemma~\ref{lem:commuting toggles} that $\tau$ and $\rho$ commute.  Fix $C\in \cC$ and let $B=\tau(C)\in \cB$.  Since $\tau$ is an involution we have 
    $$\rho(C) = \rho\tau(B) = \tau \rho(B) = \tau(B) = C.$$
    Hence $\rho$ fixes all elements in $\cC$.  As in the proof of the first part of the preceding lemma, iterating this argument implies that $\rho = 1$.    
\end{proof}

\begin{lemma}\label{lem:type 1 toggles}
Let $\sigma\in T(\cL)$ be a product of type 1 toggles.  For each block $\cB$ there exists some set $X\subseteq E$ such that for each $A\in \cB$ we have
    $$\sigma(A) = A\triangle X.$$
In particular, if  $\sigma(\cB) = \cB$ then $\sigma|_{\cB}$ is either the identity or an involution with no fixed points.  
\end{lemma}
\begin{proof}
    Let $\sigma = \tau_{x_k}\cdots \tau_{x_1}$ where each $\tau_{x_i}$ is a type 1 toggle. Set $\cB_0 = \cB$ and let $\cB_i = \tau_{x_i}(\cB_{i-1})$.  By the first part of Lemma~\ref{lem:commuting toggles} we may assume $\cB_i \neq \cB_{i-1}$.  It then follows that 
    $$\cB_{i} = \makeset{A\toggle{x_i}}{A\in \cB_{i-1}}.$$  
    Defining $X$ to be the set of all the $x_i$ that appear an odd number of times in $x_1,\ldots, x_k$ it further follows that $\sigma(A) = A\triangle X$ for every $A\in \cB$.  This proves our first claim.    The second
claim follows since $\sigma$ is now seen to  act by symmetric difference with a fixed set $X$, on all of $\cB$.
\end{proof}
\begin{lemma}\label{lem:disjoint toggles}

Let $X$ and $Y$ be disjoint subsets of $E$ and define $\sigma = \tau_{x_1}\cdots \tau_{x_m}$ and $\rho = \tau_{y_1}\cdots \tau_{y_\ell}$ where $x_i\in X, \ y_j \in Y$.  For all $A\in \cL$ we have  $\sigma(A) = \rho(A)$ if and only if $\sigma(A) =A =  \rho(A)$.  
\end{lemma}
\begin{proof}
    The reverse direction is clear.  Now assume $\sigma(A) = \rho(A)$.  For some $S\subseteq X$ and $T\subseteq Y$ we have
    $$A\triangle S = \sigma(A) =\rho(A) = A\triangle T.$$
    (Note that the subsets $S$ and $T$ may be dependent on $A$.) As $X$ and $Y$ are disjoint this can only occur if $S = T = \emptyset$, and the forward direction follows.  
\end{proof}

\begin{lemma}\label{lem: cartesian product}

Let $E_1 = \makeset{x\in E}{\tau_x \text{ is type 1}}$,\  $E_2 = \makeset{y\in E}{\tau_y \text{ is type 2}}$ and set
$\cL_i = \makeset{A\cap E_i}{A\in \cL}$.

     Suppose $T(\cL)$ is such that any product $\sigma$ of type 1 toggles  has the property that if $\sigma(\cB) = \cB$, for some block $\cB$, then $\sigma(A) = A$ for all $A\in\cB$.  Then $\cL = \cL_1\otimes \cL_2$ where $|\cL_1|$ is the number of blocks and $|\cL_2|$ is the size of each block. 
\end{lemma}
\begin{proof}
Fix $A\in \cL$ and define $\cO(A)$ to be the orbit containing $A$ under the group generated by all type 1 toggles.  As $T(\cL)$ is transitive it follows that for any two blocks there is a product of type 1 toggles that maps one to the other.  As a result $\cO(A)$ contains at least one set from each block. 
 On the other hand if $B,B'\in \cO(A)$ are in the same block it follows by our assumption about type 1 toggles that $B=B'$.  So $\cO(A)$ contains exactly one set from each block. Consequently for a fixed block $\cB$ we have the disjoint union
$$\cL = \bigcup_{A\in \cB} \cO(A).$$
We call the sets $\cO(A)$ \emph{layers}.

If $A\neq B$ are elements of the same block $\cB$ then by transitivity, the second part of Lemma~\ref{lem:commuting toggles} and the fact that $\cO(A)$ contains exactly one set from each block, it follows that for some product $\rho$ of type 2 toggles  we have $\rho(A)=B$.  From this we see that
$$A\cap E_1 = B\cap E_1 \text{ and } A\cap E_2 \neq B\cap E_2.$$
On the other hand, if $A\neq C$ are elements in the same layer then 
$$A\cap E_1 \neq C\cap E_1 \text{ and } A\cap E_2 = C\cap E_2.$$

So all the sets in $\cO(A)$ have different intersections with $E_1$ and every set $D\in \cL$ has the same intersection with $E_1$ as some set in $\cO(A)$ since $D$ is in the same block as some set in $\cO(A)$. Thus $|\cL_1|=|\cO(A)|$ is the number of blocks $\cB_i$, since $\cO(A)$ contains exactly one set from each block.  Since all sets in $\cO(A)$ have the same intersection with $E_2$ we have  
$$\cO(A) = \cL_1 \otimes \{A\cap E_2\}.$$
So we have
$$\cL = \bigcup_{A\in \cB} \cO(A) = \bigcup_{A\in \cB} \cL_1 \otimes \{A\cap E_2\} = \cL_1 \otimes \cL_2,$$ since every set in $\cL$ is in the same layer as some set in $\cB$, hence has the same intersection with $E_2$.
The last equality further shows that $|\cL_2|$ is the size of each block as claimed.
\end{proof}

\begin{lemma}\label{lem:odd block size}
Let $T(\cL)$ be imprimitive  with a system of blocks of imprimitivity $\cB_i$.  If the block size is odd, then $\cL$ is a toggle-disjoint Cartesian product of sets each of which has at least two elements.
\end{lemma}
\begin{proof}
To apply Lemma~\ref{lem: cartesian product}, suppose $\sigma$ is a product of type 1 toggles such that $\sigma(\cB_i) = \cB_i$ for some $i$.  By Lemma~\ref{lem:type 1 toggles} it follows that $\sigma|_{\cB_i}$ is either an  involution with no fixed points or the identity.  Since the blocks $\cB_i$  have odd size we see that $\sigma|_{\cB_i}$ is the identity.  Lemma~\ref{lem: cartesian product} completes the proof.
\end{proof}

\begin{lemma}\label{lem:odd cycle}
Let $T(\cL)$ be imprimitive with a system of blocks of imprimitivity $\cB_i$.  Assume there exists some $\sigma$, a product of type 1 toggles, such that $\sigma(\cB_i) = \cB_i$ and $\sigma|_{\cB_i}\neq 1$ for some $\cB_i$.  Then  $\tau|_{\cB_j}$ is an even permutation for all type 2 toggles $\tau$ and all blocks $\cB_j$.

Consequently, if for some $\cB_i$ there exists a type 2 toggle $\tau$ such that $\tau|_{\cB_i}$ is odd, then $\cL$ is a toggle-disjoint Cartesian product of two sets each of which has at least two elements. 
\end{lemma}

\begin{proof} 
Let $\cB_i$ and $\sigma$ be as stated. As $\sigma|_{\cB_i}\neq 1$, Lemma~\ref{lem:type 1 toggles} implies that $\sigma|_{\cB_i}$ is an involution with no fixed points.  Define a partition on $\cB_i$ so that elements $A$ and $B$ are in the same class if and only if $\sigma(A) = B$. Denote the classes in this partition by $\cC_j$ and observe that $|\cC_j| = 2$.    As $\sigma$ is a product of type 1 toggles, it follows from Lemma~\ref{lem:commuting toggles} that any type 2 toggle commutes with $\sigma$.  Consequently, the $\cC_j$ form a block system for the restriction to $\cB_i$ of the group $G$ generated by all the type 2 toggles. By Lemma~\ref{lem:disjoint toggles} we see that any type 2 toggle $\tau$ restricted to $\cB_i$ can either have $\tau(\cC_j) = \cC_k$ with $j\neq k$ or $\tau|_{\cC_j}=1$. As each $C_j$ has size 2 this implies that the restriction to $\cB_i$ of any type 2 toggle is an even permutation.  By the third part of Lemma~\ref{lem:commuting toggles}, this proves our first claim.

The proof of our second claim is now immediate by Lemma~\ref{lem: cartesian product}.
\end{proof}

\section{Transitive toggle groups containing short cycles}

Throughout this section we continue to assume that $T(\cL)$ is transitive. The \emph{degree} of $T(\cL)$ is the number of elements in the set $\cL$.

The first two theorems of this section state that if $T(\cL)$ has degree $n$ and contains a transposition then $T(\cL)$ is $S_n$, and if $T(\cL)$ has degree $n$ and contains a 3-cycle then $T(\cL)$ contains $A_n$.   These two theorems echo well-known results of Jordan for primitive permutation groups (see \cite[Theorem 3.3A]{dixon1996permutation} or \cite[Theorem 8.17, Corollary 8.19]{ isaacs2008finite}). Interestingly, in the toggle-group case we do not need to assume the groups are primitive but instead we obtain that as a consequence.

\begin{thm}\label{thm: transpositions}
Assume that $T(\cL)$ has degree $n$ and contains a transposition.  Then $T(\cL) \cong S_n$.
\end{thm}

\begin{proof}
We define an equivalence relation $\sim$ on $\cL$ by letting $A\sim B$ if and only if either $A=B$ or else $(A,B)\in T(\cL)$.  Since conjugating $(A,B)$ by $(B,C)$ yields $(A,C)$, it is easy to check that $\sim$ is an equivalence relation. Likewise, conjugating $(A,B)$ by any $g\in T(\cL)$ yields $(g(A),g(B))$, so the equivalence classes of $\sim$ constitute a system of blocks for $T(\cL)$. For the remainder of the proof we fix this system of blocks and denote the blocks by $\cB_1,\ldots, \cB_m$.

  We prove the theorem by showing that $m =1$ (which implies that $T(\cL)$ contains all transpositions).   By assumption we have $(A,B)\in T(\cL)$ for some $A,B\in \cL$, and we can suppose without loss of generality that $A,B\in \cB_1$.  By the second part of Lemma~\ref{lem:commuting toggles}  we can express 
  $$(A,B) = \underbrace{\tau_{y_1}\cdots \tau_{y_k}}_{\sigma}\underbrace{\tau_{x_1}\cdots \tau_{x_\ell}}_{\rho}$$
  where the  $\tau_{y_j}$ are type 1 and the $\tau_{x_i}$ are type 2.  By squaring both sides we obtain
  
  $$1 = \sigma^2 \rho^2.$$
  As $\{x_1,\ldots, x_m\} \cap \{y_1,\ldots, y_\ell\} = \emptyset$, Lemma~\ref{lem:disjoint toggles} implies that $\sigma^2 = \rho^2 = 1$.

  Assume for a contradiction that $m>1$, and let $C\in \cB_2$.  As $(A,B) = \sigma\rho$ we must have $\sigma\rho(C) = C$.  As $\sigma^2=1$ it follows that $\rho(C) = \sigma(C)$.  By Lemma~\ref{lem:disjoint toggles} then $\rho(C) = C = \sigma(C)$ for all $C\in \cB_2$.  So $\rho$, a product of type 2 toggles, fixes $\cB_2$ pointwise.  Hence by Lemma~\ref{lem:type 2 toggles}, $\rho=1$ and $\sigma = (A,B)$.  Now if each block has even size then it follows from the first part of Lemma~\ref{lem:commuting toggles} that  each $\tau_{y_i}$ is an even permutation, contradicting $\sigma= (A,B)$.  So each block has odd size. As $\sigma(\cB_1)=\cB_1$ it follows by the second claim in Lemma~\ref{lem:type 1 toggles} that $\sigma|_{\cB_1}=1$, contradicting the fact that $\sigma = (A,B)$. We conclude that $m =1$ as needed.
\end{proof}

\begin{thm}\label{thm:3-cycles}
    Assume $T(\cL)$ has degree $n$ and contains a 3-cycle.  Then $A_n\leq T(\cL)$.
\end{thm}

\begin{proof}

We define an equivalence relation $\sim$ on $\cL$ by letting  $A\sim B$ if and only if either $A=B$ or else $(A,B,C) \in T(\cL)$ for some $C$.  Using the fact that $(A,B,C)^2=(B,A,C)$, it is easy to check, as in the preceding proof,  that $\sim$ is an equivalence relation, and that the equivalence classes of $\sim$ constitute a system of blocks for $T(\cL)$.  We again denote the blocks by $\cB_1,\ldots, \cB_m$.

To prove the theorem it suffices to show that $T(\cL)$ must contain all 3-cycles or, equivalently, that $m=1$. To see that this suffices, suppose $m=1$ and let $A,B,C$ be distinct elements of $\cL$. We claim that $(A,B,C)\in T(\cL)$. Since $m=1$ we have $A\sim B$, so $(A,B,D)\in T(\cL)$ for some $D\neq A,B$.  If $D=C$ we are done.  Otherwise, $D\neq C$ and since $m=1$ we must have a 3-cycle $(D,C,E)\in T(\cL)$.  If $E\neq A,B$ we are done by conjugating $(A,B,D)$ by $(D,C,E)$.  If $E=A$ then we have $(A,B,C)=(D,C,A)(A,B,D)\in T(\cL)$.  If $E= B$ we have $(A,B,C)=(A,B,D)(D,C,B)^2\in T(\cL)$.

    We now show that $m=1$.  By assumption we have a 3-cycle $(A,B,C)$ in $T(\cL)$, and without loss of generality we can suppose  that $A,B\in \cB_1$ and therefore $C\in \cB_1$ since $(A,B,C)$ maps $A$ and $B$ into the same block. By Lemma~\ref{lem:commuting toggles} we can write
    $$(A,B,C) = \underbrace{\tau_{y_1}\cdots \tau_{y_k}}_{\sigma}\underbrace{\tau_{x_1}\cdots \tau_{x_\ell}}_{\rho}$$
  where the  $\tau_{y_j}$ are type 1 and the $\tau_{x_i}$ are type 2.

For a contradiction assume $m>1$ and let $D\in \cB_i$ where $i\neq 1$.  As $(A,B,C) = \sigma\rho$ we have $\sigma\rho(D) = D$ and hence $\rho(D) = \sigma^{-1}(D)$.  By Lemma~\ref{lem:disjoint toggles} we conclude  that $\rho(D) = D$.  As $D$ is an arbitrary element of $\cB_i$ then $\rho$ is the identity on $\cB_i$.  By Lemma~\ref{lem:type 2 toggles} it follows that $\rho=1$.  So $\sigma =(A,B,C)$. But then Lemma~\ref{lem:type 1 toggles} forces $\sigma^2=1$, a contradiction.  We conclude that $m=1$ as needed.
\end{proof}

In addition to the results that motivated Theorem~\ref{thm: transpositions} and Theorem~\ref{thm:3-cycles}, Jordan also proved that if a primitive permutation group $G$ of degree $n$ contains a cycle of prime length $p\leq n-3$ then $A_n\leq G$ (see \cite[Theorem 3.3E]{dixon1996permutation} or \cite[Theorem 8.23]{ isaacs2008finite}). 
Much more recently, Jones, in \cite[Corollary 1.3]{jones2014primitive}, removed the assumption that the length of the cycle be prime in Jordan's result.  To obtain an analogue of Jones' result for toggle groups (as opposed to primitive permutation groups), we will combine Jones' result with the following.

\begin{thm}\label{thm:primitivity}
Suppose $T(\cL)$ has degree $n$ and contains a nontrivial cycle $\gamma$ of length  $ \leq n-1$.  Then $T(\cL)$ is primitive.

\end{thm}

\begin{proof}
Suppose not, and take a system of blocks of imprimitivity for $T(\cL)$. Note that if two points of $\cL$ are in the same block then so are their images under $\gamma$, so if one point moved by $\gamma$ is in the same block as a fixed point of $\gamma$ then all points moved by $\gamma$ are in that block.

We first consider the case when all points moved by $\gamma$ are in a block $\cB$ that contains a fixed point of $\gamma$.  We know there are blocks other than $\cB$, and each of these other blocks consists entirely of fixed points of $\gamma$. Let $\cC$ be any one of these other blocks.  If we write $\gamma=\sigma\rho$ with $\sigma$ a product of type 1 toggles and $\rho$ a product of type 2 toggles, then we can argue as we have before that $\rho$ must fix every point of $\cC$, and therefore $\rho=1$ by Lemma~\ref{lem:type 2 toggles}. Therefore $\gamma=\sigma$, and $\gamma$ maps $\cB$ to itself. By Lemma~\ref{lem:type 1 toggles}, $\gamma^2=1$, so $\gamma$ is a transposition. By Theorem~\ref{thm: transpositions}, $T(\cL)$ is $S_n$, so since $T(\cL)$ has degree $n$, $T(\cL)$ is primitive.

We now consider the case when no points moved by $\gamma$ are in the same block as any fixed point of $\gamma$.  We first show that in this case each block has even size.  If not, then there is a block $\cB$ of odd size $|\cB|> 1$ consisting of points moved by $\gamma$, and there is a block $\cC$ consisting of fixed points of $\gamma$ (since $\gamma$ has length $\leq n-1$).  As in the first case, we write $\gamma=\sigma\rho$ and use the existence of the block $\cC$ to show that $\rho=1$ and $\gamma=\sigma$.  If we choose $A\neq B$ in $\cB$ then there exists a positive integer $m$ such that $\gamma^m(A)=B$, and therefore $\gamma^m$ maps $\cB$  to $\cB$.  Since $\gamma^m=\sigma^m$, a product of type 1 toggles, Lemma~\ref{lem:type 1 toggles} implies that $\gamma^m|_{\cB_i}$ is either an  involution with no fixed points or the identity.  Since the blocks have odd size we see that $\sigma|_{\cB_i}$ is the identity.  But this is impossible, since $\gamma^m(A)=B$ where $A\neq B$.

So we know the blocks have even size.  Since the elements moved by $\gamma$ comprise a set of blocks, $\gamma$ has even length, hence is an odd permutation. But the fact that the blocks have even size also implies, by the first part of Lemma~\ref{lem:commuting toggles}, that every type 1 toggle is an even permutation. Since $\gamma=\sigma$ is a product of type 1 toggles, $\gamma$ is an even permutation.  This contradiction concludes the proof.

\end{proof}

\begin{corollary}\label{corollary:includes A_n}
If $T(\cL)$ has degree $n$ and contains a nontrivial cycle of length $\leq n-3$ then $A_n\leq T(\cL)$.

\end{corollary}
\begin{proof}
By Theorem~\ref{thm:primitivity}, $T(\cL)$ is primitive, so the corollary follows from \cite[Corollary 1.3]{jones2014primitive}, which states, among other things, that if $G$ is a primitive permutation group of degree $n$ that contains a nontrivial cycle with at least three fixed points, then $A_n\leq G$.

\end{proof}

Corollary~\ref{corollary:includes A_n} provides an alternate proof of Theorem~\ref{thm:3-cycles} if we check separately the cases $n\leq 5$.  But Corollary~\ref{corollary:includes A_n} relies on \cite[Corollary 1.3]{jones2014primitive}, which depends on the classification of the finite simple groups.

\section{Imprimitive toggle groups containing long cycles}

We have shown that a toggle group that contains a short cycle must be primitive.  Therefore, if an imprimitive toggle group contains a cycle that cycle can only be a long cycle.  In this section we study imprimitive toggle groups that contain a long cycle.

\begin{lemma}\label{lem:cartesian long cycle}
Let $E_1, E_2$ be disjoint ground sets and let $\cL_i\subseteq 2^{E_i}$ be such that $|\cL_i|
\geq 2$ for $i=1,2$. Set $\ell = |\cL_1|$ and $m = |\cL_2|$ and $\cL = \cL_1\otimes \cL_2$.  Then $T(\cL)$ contains a long cycle if and only if $(\ell,m) = 1$ and both $T(\cL_1)$ and $T(\cL_2)$  contain long cycles.
\end{lemma}
\begin{proof}
 Let $\cL_1 = \{A_1,\ldots, A_\ell\}$ and $\cL_2= \{B_1,\ldots, B_m\}$.  Consider the block system given by the blocks $$\cC_i = \makeset{A_i \cup B_j}{B_j \in \cL_2},$$
so that the type 1 toggles are precisely $\tau_x$ for $x\in E_1$ and the type 2 toggles are precisely $\tau_y$ for $y\in E_2$. (Recall that, by the convention we adopted at the outset, no $\tau_x$ or $\tau_y$ is the identity.)  Note that $|\cC_i|\geq 2$ for all $\cC_i$. 

   Suppose $T(\cL)$ contains a long cycle $\gamma$. Fix a block $\cC$. Since $|\cC|\geq 2$ and $\gamma$ is a long cycle, there exists a smallest positive integer $a$ such that $\gamma^a$ maps a point of $\cC$ into $\cC$. Choose $X\in \cC$ such that $\gamma^a(X)\in \cC$. By our choice of $a$, it follows that $X,\gamma(X), \gamma^2(X),\ldots, \gamma^{(a-1)}(X)$ must all be in distinct blocks and so $a\leq \ell$.  Further observe that $\gamma^{ka}(X)\in \cC$ for all $k$.  If we write $\ell m=ka+r$, with $0\leq r< a$,  then $\gamma^r(\gamma^{ka}(X))=X\in \cC$ since $\gamma$ is an $\ell m$-cycle. Since $0\leq r <a$ we must have $r=0$ by our choice of $a$.  So $a|\ell m$. Again since $\gamma$ is an $\ell m$-cycle, we see that $\gamma^{a}(X),\gamma^{2a}(X)\ldots, \gamma^{\ell m}(X)$ are distinct and in $\cC$.  As  $|\cC|=m$ have $\ell m/a \leq m$, so  $\ell\leq a$ and therefore $\ell=a$ and we have
   $$\gamma^{\ell} = \gamma_1\cdots \gamma_{\ell},$$
   where $\gamma_i$ is a long cycle on block $\cC_i$.

   Now, using Lemma~\ref{lem:commuting toggles}, write $\gamma = \sigma\rho$ where $\sigma$ is a product of type 1 toggles and $\rho$ is a product of type 2 toggles. Since $\sigma$ and $\rho$ commute we have $\gamma^{\ell} = \sigma^{\ell} \rho^{\ell}$, and   since $\gamma^{\ell}$ and $\rho^{\ell}$ are type 2 permutations so is $\sigma^{\ell}$.  Since $\sigma^{\ell}$ is a product of type 1 toggles and $\cL$ is a toggle-disjoint Cartesian product it follows that $\sigma^{\ell} = 1$.  So 
   $$\gamma_1\cdots \gamma_{\ell} = \rho^{\ell}.$$
   Restricting our attention to a fixed block $\cC_i$ we see that $\rho^{\ell}|_{\cC_i}  = \gamma_i$.  Since $\gamma_i$ is a long cycle on this block it follows that $\rho|_{\cC_i}$ must be a long cycle on $\cC_i$ and that $(m,\ell) = 1$.  It is now immediate that $T(\cL_2)$ must contain a long cycle as well.  By repeating this argument with the block system
   $$\cC_i' = \makeset{B_i \cup A_j}{A_j \in \cL_1}$$
   we conclude that $T(\cL_1)$ must also have a long cycle as claimed.  

Let us now consider the converse. Suppose $(m,\ell) = 1$ and $T(\cL_1)$ and $T(\cL_2)$  have long cycles $\gamma$ and $\delta$, respectively.  Note that $\gamma$ and $\delta$ have orders $\ell$ and $m$, respectively.  Label the sets in $\cL$ by $(i,j)$ for $1\leq i\leq \ell, \ 1\leq j\leq m$ and observe that $T(\cL)$ contains permutations 
$$\gamma'(i,j) : = (\gamma(i),j)$$
and
$$\delta'(i,j) : = (i,\delta(j))$$
with orders $\ell$ and $m$, respectively.  The element $\gamma'\sigma'$ is a long cycle in $T(\cL)$, because $\gamma'\sigma'(1,1),\ldots, (\gamma'\sigma')^{\ell m}(1,1)$ gives us all the elements of $\cL$.
\medskip
\end{proof}

\begin{thm}\label{thm:imprimitive with long cycle}
    Assume $T(\cL)$ is imprimitive and contains a long cycle.  Then $\cL = \cL_1\otimes \cL_2$ for some $\cL_1$ and $\cL_2$ such that $|\cL_1|\geq 2, \ |\cL_2|\geq 2,\ (|\cL_1|, |\cL_2|) =1$, and each of $T(\cL_1)$ and $T(\cL_2)$ contains a long cycle.
\end{thm}
\begin{proof}
As $T(\cL)$ is imprimitive, let $\cB_1,\ldots, \cB_{\ell}$ be a system of blocks of imprimitivity.  Set $|\cB_i| = m$.  If $m$ is odd, then the result follows from Lemma~\ref{lem:odd block size} and Lemma~\ref{lem:cartesian long cycle}.

Now assume $m$ is even. If we replace the blocks $\cC_i$ in the proof of Lemma~\ref{lem:cartesian long cycle} by the $\cB_i$'s and define $a$ as we did before, we obtain $$\gamma^{\ell} = \gamma_1\cdots \gamma_{\ell},$$
   where $\gamma_i$ is a long cycle on block $\cB_i$.  We write $\gamma=\sigma\rho$ as before and have $$\gamma_1\cdots \gamma_{\ell}=\gamma^{\ell}=\sigma^{\ell}\rho^{\ell}.$$ Since $\gamma^{\ell}$ and $\rho^{\ell}$ are type 2 permutations, so is $\sigma^{\ell}$.

We claim that $\sigma^{\ell} =1$.  For a contradiction, assume otherwise.  Lemma~\ref{lem:odd cycle} then implies that every type 2 toggle is even.  Since $m$ is even we also have (using the first part of Lemma~\ref{lem:commuting toggles})  that every type 1 toggle is even.  Therefore since $\gamma$ is product of toggles it must be even, but $\gamma$ is an $\ell m$-cycle which is odd  because $m$ is even.  We conclude that $\sigma^{\ell} = 1$ as claimed.    

Since $\sigma^{\ell} = 1$ we see that  $\rho^{\ell}= \gamma^{\ell}$ which is the long cycle $\gamma_i$ on each block $\cB_i$.  So $\rho$ must be a long cycle on each block $\cB_i$.  As each block has even size $\rho|_{\cB_i}$ must be odd.  Since $\rho$ is a product of type 2 toggles we must have, for each $\cB_i$, some type 2 toggle $\tau$ such that $\tau|_{\cB_i}$ is odd. 
 Lemma~\ref{lem:odd cycle} and Lemma~\ref{lem:cartesian long cycle} now complete the proof.
\end{proof}

\begin{corollary}\label{corollary:factorization}
Assume $T(\cL)$ is imprimitive and contains a long cycle. Then $$T(\cL) \cong T(\cL_1) \times \cdots \times T(\cL_k),$$
where each $T(\cL_i)$ is primitive and contains a long cycle, their degrees are larger than 1 and pairwise coprime, and $|\cL|=|\cL_1|\cdots|\cL_k|$.
\end{corollary}

\begin{proof}
Apply Theorem~\ref{thm:imprimitive with long cycle} repeatedly.

\end{proof}

\begin{corollary}\label{corollary:primepowerdegree}
If $T(\cL)$ is transitive, has prime-power degree, and contains a long cycle, then $T(\cL)$ is primitive.

\end{corollary}

\begin{proof}
This is immediate from Corollary \ref{corollary:factorization}.

\end{proof}

Corollary~\ref{corollary:primepowerdegree} expresses a special property of  toggle groups.  For example, the dihedral group of order eight, in its natural action on a set of order four, contains a 4-cycle but is not primitive by \cite[Theorem 4.2A(vi)]{dixon1996permutation}, since it has nontrivial center and is nonabelian.

In view of Theorem~\ref{thm:primitivity}, Corollary~\ref{corollary:includes A_n} and Corollary~\ref{corollary:factorization} it would be very interesting to have an answer to the following question: Which primitive toggle groups of degree $n$ contain a cycle of length at least $n-2$ but do not contain $A_n$?   A list of all the primitive permutation groups that contain a cycle of length at least $n-2$ but do not contain $A_n$ is given in \cite[Theorem 1.2]{jones2014primitive}.  The first groups on the list are the subgroups of the affine group $AGL_1(p)$ that contain a cyclic subgroup of order $p$, and we can show that these are not toggle groups.  But at this point we do not know the status of the other groups on the list.           

\bibliographystyle{acm} 
\bibliography{mybib}
\end{document}